\documentclass[10pt,a4paper,reqno]{amsart}
\usepackage{amsfonts,amsthm,latexsym,amsmath,amssymb,amscd,amsmath, epsf}
\usepackage{url,graphics} 
\usepackage[dvips]{graphicx,epsfig} 
\usepackage{caption}

\input cyracc.def
\newfam\cyrfam
\font\tencyr=wncyr10
\font\sevencyr=wncyr7

\textfont\cyrfam=\tencyr \scriptfont\cyrfam=\sevencyr
  \scriptscriptfont\cyrfam=\sevencyr

\font\tencyr=wncyr10

\font\tencyri=wncyi10

\newtheorem{definition}{Definition}
\newtheorem{theorem}{Theorem}
\newtheorem{lemma}{Lemma}

\newtheorem{proposition}{Proposition}
\newtheorem{remark} {Remark}
\newtheorem{example} {Example}
\newtheorem{problem} {Problem}

\newtheorem{conjecture}[problem]{Conjecture}

\newcommand {\bC} {\mathbb {C}}

\newcommand {\bZ} {\mathbb Z}

\newcommand {\CC} {\mathcal C} 
 
\newcommand {\cD} {\mathcal {D}}
\newcommand {\J} {\mathcal {J}}  
\newcommand {\I} {\mathcal {I}}  
\newcommand {\B} {\mathcal {B}}  
\newcommand {\K} {\mathbb {K}}  

\newcommand {\KK} {\mathcal {K}}
\newcommand {\CA} {\mathcal A}
\newcommand{\CU}{\mathcal U}
\newcommand{\cF}{\mathcal F}

\newcommand{\tY}{\widetilde{Y}}
\newcommand{\tZ}{\widetilde{Z}}

\begin{document}

             \title[``K-theoretic" analog distinguishes graphs]
             {``K-theoretic" analog of Postnikov-Shapiro algebra distinguishes graphs}

\author [G.~Nenashev]{Gleb Nenashev}
\address{Department of Mathematics, Stockholm University, SE-106 91, Stockholm,
            Sweden}
\email{nenashev@math.su.se}

\author[B.~Shapiro]{Boris Shapiro}
\address{Department of Mathematics, Stockholm University, SE-106 91, Stockholm,
            Sweden}
\email{shapiro@math.su.se}

\date{\today}
\keywords{spanning forests and trees, commutative algebras, filtered algebras}
\subjclass[2000]{Primary 05E40, Secondary 05C31}

\dedicatory{\tencyri Izvestnomu yaponskomu matematiku Tole Kirillovu posvyashchaet{}sya}

\begin{abstract}  In this paper we study a filtered ``$K$-theoretical"  analog of a  graded algebra  associated to any  loopless  graph $G$  which was introduced  in \cite{PS}. We show that two such filtered algebras are isomorphic if and only if their graphs    are isomorphic. We also study a large family of filtered generalizations of the latter graded algebra which includes the above  ``$K$-theoretical"  analog. 


\end{abstract}

\maketitle

\section {Introduction}

The following  square-free algebra $\CC_G$    associated to an arbitrary vertex labeled graph $G$ was defined in  \cite {PS}, see also \cite{AP}.  
 Let $G$ be a graph without loops on the vertex set $\{0, . . . , n\}.$ (Below we always  assume that all graphs might have multiple edges, but no loops). 
 Throughout the whole paper, we fix  a field $\K$ of zero characteristic.    Let  $\Phi_G$ be the graded commutative algebra over $\K$ 
generated by the variables $\phi_e, e \in G$, with the defining relations:
$$(\phi_e)^2 = 0, \quad \text {for any edge}\; e\in G.$$
Let $\CC_G$ be the subalgebra of  $\Phi_G$ generated by the elements
$$X_i =\sum_{e\in G} c_{i,e} \phi_e,$$
for  $i = 1, . . . , n, $ where 
\begin{equation}\label{eq:def}
c_{i,e}=\begin{cases} \;\;\;1\quad \text{if}\; e=(i,j), i<j;\\
                                            -1\quad\text{if}\; e=(i,j), i>j;\\
                                             \;\;\;0\quad \text{otherwise}.
\end{cases}
\end{equation}

For the reasons which will be clear soon, we call $\CC_G$ the {\it spanning forests counting algebra}  of $G$. Its Hilbert series and the set of defining relations were calculated in  \cite {PSS} following the initial paper \cite{SS}. Namely, let $\J_G$ be the ideal in $\K[x_1,\dots,x_n]$ generated by the polynomials
\begin{equation}\label{eq:relF}
p_I=\left( \sum_{i\in I}x_i\right)^{D_I+1},
\end{equation} 
where $I$ ranges over all nonempty subsets in $\{1,\dots,n\}$ and $D_I=\sum_{i\in I}d_I(i),$ where $d_I(i)$ is  the total number of edges connecting a given vertex  $i\in I$ with all vertices outside $I$. Thus, $D_I$ is the total number of edges between $I$ and the  complementary set of vertices $\bar I$. Set $\B_G:=\K[x_1,\dots,x_n]/\J_G.$

\begin{remark}{\rm 
Observe that since $\sum_{i=0}^nX_i=0$, we can  define $\CC_G$ as  the subalgebra of  $\Phi_G$ generated by $X_0,X_1,\dots,X_n$. 

We  can also define $B_G$ as the quotient algebra of $\K[x_0,\dots,x_n]$ by the ideal generated by $p_I$, where $I$ runs over all subsets of $\{x_0,x_1,\cdots,x_n\}$. This follows from the relation 
$$p_I=\left( \sum_{i\in I}x_i\right)^{D_I+1}=\left( p_{\{0,1,\dots,n\}}-\sum_{i\in \overline{I}}x_i\right)^{D_{\overline{I}}+1}.$$
} 
\end{remark}

To describe the Hilbert polynomial of $\CC_G,$  we need the following classical notion going back to W.~T.~Tutte. Given a simple graph $G,$ fix an arbitrary linear order of its edges. Now, given a spanning forest $F$ in $G$ (i.e., a subgraph without cycles which includes all vertices of $G$) and  an edge $e\in G\setminus F$ in its complement, we say that $e$ is {\it externally active} for $F,$ if there exists a cycle $C$ in $G$ such that all edges in $C\setminus \{e\}$ belong to $F$ and $e$ is minimal in $C$ with respect to the chosen linear order. The total number of external edges is called the {\it external activity} of $F$.  Although the external activity of a given forest/tree in $G$ depends on the choice of a linear ordering of edges, the total number of forests/trees with a given external activity is independent of this ordering. Now we are ready to formulate the main result of~\cite{PSS}. 

\begin{theorem}\label{th:forests}{\rm [Theorems~3 and 4 of \cite{PSS}]} For any simple graph $G,$  algebras  $\B_G$ and $\CC_G$ are isomorphic. 
The total dimension of these algebras (as vector spaces over $\K$) is equal to the number of spanning subforests in~$G$.
The dimension  of the $k$-th graded component of these algebras equals
the number of subforests $F$ in $G$ with external activity $|G|-|F|-k$. Here  $|G|$ (resp. $|F|$) stands for  the number of edges in $G$ (resp. $F$). \end{theorem}\

In the above notation, our main object will be the filtered subalgebra $\KK_G\subset \Phi_G$    defined by the generators: 
$$Y_i=\exp(X_i)=\prod_{e\in G} (1+c_{i,e}\phi_e),\; i=0,\dots,n.$$
(Notice that we have one more generator here than in the previous case.)

\begin{remark} {\rm Since $Y_i$ is  obtained by exponentiation of $X_i,$ we call $\KK_G$ the ``K-theoretic" analog of $\CC_G$. }
\end{remark}

\medskip
Our first result is as follows.  Define  the ideal  $\I_G$ in $\K[y_0,y_1,\dots,y_n]$ as generated by the polynomials
\begin{equation}\label{eq:relFF}
q_I=\left( \prod_{i\in I}y_i-1\right)^{D_I+1},
\end{equation} 
where $I$ ranges over all nonempty subsets in $\{0,1,\dots,n\}$ and the number $D_I$ is the same as in 
\eqref{eq:relF}. Set $\cD_G:=\K[y_0,\dots,y_n]/\I_G.$ 

\begin{theorem}\label{th:4alg} For any graph $G,$  algebras   $B_G,$  $\CC_G$, $\cD_G$ and $\KK_G$  are  isomorphic as (non-filtered) algebras.
\end{theorem}

Moreover, the following stronger statement holds.

\begin{theorem}\label{th:main} For any graph $G$,  algebras  $\cD_G$ and $\KK_G$ are  isomorphic  as filtered algebras. 
\end{theorem} 


Recall that in a recent paper \cite{Ne}  the first author has shown that $\CC_G$ contains all information about the matroid of $G$ and only it. Namely, 

\begin{theorem}\label{nenashev}{\rm [Theorem~5 of \cite{Ne}]} Given two graphs $G_1$ and $G_2,$ algebras $\CC_{G_1}$ and $\CC_{G_2}$ are isomorphic  if and only if the matroids of $G_1$ and $G_2$ coincide. (The  latter isomorphism can be thought of either as graded or as non-graded, the statement holds in both cases.) 
\end{theorem}

On the other hand, filtered algebras $\cD_G$ and $\KK_G$ contain complete information about $G$. 

\medskip
\begin{theorem}\label{th:Isom} Given two graphs $G_1$ and $G_2$ without isolated vertices, $\KK_{G_1}$ and $\KK_{G_2}$ are isomorphic as filtered algebras  if and only if  $G_1$ and $G_2$ are isomorphic.

\end{theorem}

The structure of this paper is as follows.  In \S~2 we prove new results formulated above. In \S~3 we discuss Hilbert series of similar algebras defined by other sets of generators. In  \S~4  we discuss ``K-theoretic" analogs of algebras counting spanning trees. Finally, in \S~5 we present a number of open problems.

\medskip 
\noindent 
{\it Acknowledgements.} The study of $\cD_G$ and $\KK_G$  was initiated by the second author jointly with Professor A.~N.~Kirillov during his visit to Stockholm  in October-November 2009 supported by  the Swedish Royal Academy.  The second author is happy to acknowledge the importance of this visit for the present project and to dedicate this paper to Professor Kirillov.

\section {Proofs}


To prove Theorem~\ref{th:4alg}, we need some preliminary results.

\begin{lemma}\label{lem:isCK}
For any simple graph $G,$  algebras   $\CC_G$ and $\KK_G$  coincide as subalgebras of $\Phi_G$.
\end{lemma}

\begin{proof}
Since $(X_i)^{d_i+1}=0,$ where $d_i$ is the degree of vertex $i$, then 
$$Y_i=\exp(X_i)=1+\sum_{j=1}^{d_i}\frac{(X_i)^j}{j!}.$$ Hence $Y_i\in \CC_G$ which means that $\KK_G\subset \CC_G \subset \Phi_G$.

To prove the opposite inclusion, consider $\widetilde{Y}_i=Y_i-1 =\exp (X_i) -1$. Since $X_i|\widetilde{Y}_i$, we get $$(\tY_i)^{d_i+1}=0.$$ 
Using the relation $X_i=\ln(1+\widetilde{Y}_i)=\sum_{j=1}^{d_i} \frac{(-1)^{j-1}(\widetilde{Y}_i)^j}{j!}$, we conclude $X_i\in\KK_G$.
Thus  $\CC_G\subset \KK_G$, implying that  $\CC_G$ and $\KK_G$ coincide.
\end{proof}

\begin{lemma} \label{lem:isBD}
For any simple graph $G,$  algebras   $B_G$  and $\cD_G$  are  isomorphic as (non-filtered) algebras.
\end{lemma}

\begin{proof}
First we change the variables  in $\cD_G$ by using $\widetilde{y_i}=y_i-1$,   $i=0,1,\dots,n$.
The generators of  ideal $\I_G$  transform as  $$\widetilde{q}_I=\left( \prod_{i\in I}(\widetilde{y_i}+1)-1\right)^{D_I+1},$$
for any subset $I\subset \{0,1,\dots,n\}$.

Since for every vertex $i=0,1,\dots,n,$  $$((\widetilde{y_i}+1)-1)^{d_i+1}=\widetilde{y_i}^{d_i+1},$$
we can consider $\cD_G$ as the quotient $\K[[\widetilde{y_0},\dots,\widetilde{y_n}]]/\widetilde{\I}_G$  of the ring of formal power series factored by the ideal $\widetilde{\I}_G$   generated by all $\widetilde{q}_I$.

Similarly we can consider $B_G$ as the quotient $\K[[x_0,\dots,x_n]]/\widetilde{\J}_G$  of the ring of formal power series by the  ideal $\widetilde{\J}_G$  generated by all $p_I$.

Introduce the homomorphism $\psi: \; \K[[\widetilde{y}_0,\dots,\widetilde{y}_n]] \mapsto K[[x_0,\dots,x_n]]$ defined by:  
$$\psi:\widetilde{y}_i\to {e^{x_i}-1}.$$ In fact, $\psi$ is an isomorphism, because $\psi^{-1}$ is defined by $x_i\to \ln(1+\widetilde{y}_i)$.

\smallskip
Let us look at  what happens with the ideal $\widetilde{\I}_G$ under the action of $\psi$.  For a given  $I\subset \{0,1,\dots,n\}$, consider the  generator $\widetilde{q}_I.$ Then,  
$$
\psi(\widetilde{q}_I)=\left( \prod_{i\in I}(\psi(\widetilde{y_i})+1)-1\right)^{D_I+1}=\left( \prod_{i\in I}e^{x_i}-1\right)^{D_I+1}=$$

$$=\left( \exp \left(\sum_{i\in I} x_i\right)-1\right)^{D_I+1}=\left(\sum_{i\in I} x_i\right)^{D_I+1}\cdot \left( \frac{\exp \left(\sum_{i\in I} x_i\right)-1}{\sum_{i\in I} x_i}\right)^{D_I+1}.
$$
The factor $\frac{\exp \left(\sum_{i\in I} x_i\right)-1}{\sum_{i\in I} x_i}$ is a formal power series starting with the constant term $1$. Hence the last factor in the right-hand side of the latter expression  is an invertible power series. Thus, the generator $\widetilde{q}_I$ is mapped by $\psi$ to the product $p_I\cdot *,$ where $*$ is an invertible series. This  implies $\psi(\widetilde{\I}_G)=\widetilde{\J}_G$. Hence algebras $\cD_G$ and $B_G$ are isomorphic.
\end{proof}\

\begin{proof}[Proof of Theorem~\ref{th:4alg}] 
By Lemmas~\ref{lem:isCK},~\ref{lem:isBD} and Theorem~\ref{th:forests}, we get that all four  algebras are isomorphic to each other. Furthermore, by Theorem~\ref{th:forests}, we know that their total dimension over $\K$ is the number of subforests in $G$.
\end{proof}\

Theorem~\ref{th:main} now follows from Theorem~\ref{th:4alg}.
\begin{proof}[Proof of Theorem~\ref{th:main}] 
Consider the surjective homomorphism $h:\cD_G\to \KK_G$, defined by:  
$$h(y_i)=Y_i,\; i=0,1,\dots,  n.$$ 
(It is indeed a  homomorphism because every relation $q_I$ holds for $Y_0,\dots,Y_n$.) By Theorem~\ref{th:4alg} we know that these algebras have the same dimension, implying that $h$ is an isomorphism. It is clear that $h$ preserves the filtration.
\end{proof}\

\subsection{Proving Theorem~\ref{th:Isom}} 


We start with a few definitions.
 
Given a commutative algebra $A$, its element $t\in A$ is called {\it reducible nilpotent} if and only if there exists a presentation $t=\sum u_iv_i$, where all $u_i,v_i$ are nilpotents.

For a nilpotent element  $t\in A$, define its {\it degree} $d(t)$   as the minimal non-negative integer for which there exists a reducible nilpotent element $h\in A$ such that $$(t-h)^{d+1}=0.$$

Given an element  $R\in \Phi_G$, we say that an edge-element $\phi_e$ {\it belongs to $R$}, if monomial $\phi_e$ has a non-zero  coefficient in the expansion of $R$ as the sum of square-free monomials in $\Phi_G$. 

\begin{lemma}
\label{lem-degree}
For any nilpotent element $R\in \KK_G \subset \Phi_{G}$,  the degree $d(R)$ of  $R$ equals the  number of edges of $G$ belonging to $R$.
\end{lemma}
\begin{proof}
We can write $R$ in terms of $\{X_0,\ldots,X_n\}$. (Observe that $\KK_G$ and $\CC_G$ coincide as subsets of $\Phi_G$,  but have different graded/filtered structures). Now we can concentrate on the graded structure of $\CC_G$.  Select the part of $R$ which lies in the first graded component of $\CC_G$. Thus $$R=R_1+R'=\sum_{i=0}^na_iX_i+R',$$ where $R'$ is reducible nilpotent because it belongs to the linear span of other graded components.   
  Thus $d(R)=d(R_1)$. The statement of Lemma~\ref{lem-degree} is obvious for $R_1.$ Additionally   by construction,  an edge-element $\phi_e$ belongs to $R$ if and only if it belongs to~$R_1$.
\end{proof}\


\begin{lemma}
\label{pr-basis}
Given  a graph $G$, let $\{ \tY_0,\ldots,\tY_n \}$ be the set of generators of $\KK_G$ corresponding to the vertices (i.e., $\widetilde{Y}_i=\exp(X_i)-1$). Then 
\begin{enumerate}
\item $\{\tY_0,\ldots,\tY_n\}$ are nilpotents;
\item $\sum_{i=0}^n \ln(1+\tY_i)=0$;
\item for any subset $I\subset [0,n]$ and any set of pairwise distinct non-zero numbers $a_i\in \K$ ($i\in I$), 
the degree $d(\sum_{i\in I} a_i \widetilde{Y_i})$ is equal to the number of edges incident to the vertices belonging to $I$;
\item the number of edges between vertices $i$ and $j$ equals to $\frac{d(\tY_i)+d(\tY_j)-d(\tY_i+\tY_j)}{2}$.
\end{enumerate}
\end{lemma}
\begin{proof}

Item (1) is obvious. 

To settle (2), observe that $\ln(1+\widetilde{Y_i})=X_i$ which implies   
$$\sum_{i=0}^n \ln(1+\widetilde{Y_i})=\sum_{i=0}^n X_i=0.$$

To prove (3), notice that, by  Lemma~\ref{lem-degree}, the degree $d(\sum_{i\in I} a_i \widetilde{Y}_i)$ is equal to the number of edges belonging to the sum $\sum_{i\in I} a_i \widetilde{Y}_i$. 
Each edge belongs  either to zero, to one or to two generators $\widetilde{Y}_i$ from the latter sum. Moreover,  if an edge belongs  to two generators,  then it has  coefficients of opposite signs. Since all $a_i$ are different,  an edge-element $\phi_e$ belongs to $\sum_{i\in I} a_i \widetilde{Y}_i$ if and only if it belongs to at least one $\widetilde{Y}_i,$ for $i\in I$. 
Thus the degree $d(\sum_{i\in I} a_i \widetilde{Y}_i)$ is the number of edges incident to all vertices from $I$.

To settle (4), notice that if $e$ is an edge between vertices $i$ and $j$, then $\phi_e$ belongs to $\tY_i$ and to $\tY_j$ with the opposite coefficients. Therefore $\phi_e$  does not belong to $(\tY_i+\tY_j)$. 
Using Lemma~\ref{lem-degree}, we get that $d(\tY_i)+d(\tY_j)-d(\tY_i+\tY_j)$ equals twice the number of edges between $i$ and $j$.
\end{proof}\

Our proof of Theorem~\ref{th:Isom} uses the following technical lemma which should be obvious to the specialists. 
\begin{lemma}[Folklore] 
\label{gr-edges}
Let $E$ be the set of  edges of some graph $G$  without isolated vertices.
If we know the following information: 
\begin{enumerate}
\item  which pairs $e_i,e_j\in E$ of edges are multiple, i.e., connect the same pair of vertices;
\item  which pairs $e_i,e_j\in E$ of edges have exactly one common vertex;
\item which triples $e_i,e_j,e_k\in E$ of edges form a triangle,
\end{enumerate}
 then we can reconstruct $G$ up to an isomorphism.
\end{lemma}
\begin{proof}
Assume the contrary, i.e.,  that there exist  two non-isomorphic graphs $G$ and $G'$ such that there exists  a bijection $\psi$ of their edge sets $E$ and $E'$ preserving (1) - (3). 
 Assume that under this bijection an  edge $e\in E$ corresponds to the edge $e'\in E'$.
Additionally assume that $|V(G')|\geq |V(G)|$.

Now we construct an isomorphism between $G$ and $G'$. Let us split the vertices of  $G$  into two subsets: $V(G)=\widehat{V}(G)\cup \tilde{V}(G),$ where $\widehat{V}(G)$ are all vertices which are incident to some pair of non-multiple edges.

Let us construct a bijection  $\psi$ between the vertices of $G$ and $G'$, which extends the given bijection $\psi$ of edges, i.e., for any $e=uv=\in E$, $e'=\psi(e)=\psi(u)\psi(v).$

At first  we define it on $\widehat{V}(G)$. Namely,  given a  vertex $v\in \widehat{V}(G),$  choose two non-multiple edges  $e_i$ and $e_j$ incident to it, and define $\psi(v)$ as a common vertex of $e_i'$ and $e_j'.$
We need to show that  $\psi(v)$ does not depend on the  choice of $e_i$ and $e_j$.  It is enough to check it for  a pair $e_i$ and $e_k\neq e_j$, where $e_k$ is another edge incident to $v$.
Indeed, if $e_k'$ has no common vertex with both $e'_i$ and $e'_j$, then $e'_i$, $e'_j$ and $e'_k$ form a triangle in $G'$ (because $e'_k$ has a common vertex with $e'_i$ and with $e'_j$). Hence,  $e_i$, $e_j$ and $e_k$ form a triangle in $G$, but they  have a common vertex $v$. Contradiction. 

Now we need to extend  $\psi$ to vertices belonging to  $\tilde{V}(G_1)$. Note that each vertex $v\in \tilde{V}(G_1)$ has exactly one adjacent vertex. There are two possibilities. 

{$1^\circ$} {\it Adjacent vertex $u$ of $v$ belongs to $\widehat{V}(G)$.} Consider the edge $e_{uv}\in E.$  (There might be several such edges, but this is not important, because in $G'$ they are also multiple.) Knowing the image $\psi(e_{uv})$ and the vertex $\psi(u)$, we  define $\psi(v)$ as the vertex of $\psi(e_{uv})$ different from $\psi(u)$.

{$2^\circ$} {\it Adjacent vertex $u$ of $v$ belongs to $\tilde{V}(G)$.} Consider edge $e_{uv}\in E$ 
Knowing $\psi(e_{uv})$, we define $\psi(u)$ and $\psi(v)$ as the vertices of edge~$\psi(e_{uv})$ (not important which is mapped  to which).

Since $G'$ has no isolated vertices and each edge $e'$ has exactly  two incident vertices from $\psi(V)$, we get that 
$\psi:G\to G'$ is surjective. Hence, $\psi:G\to G'$ is an isomorphism (otherwise it must be non-injective on vertices and, hence, $|V(G)|>|V(G')|$). Therefore $G$ and $G'$ are isomorphic.
\end{proof}

\begin{proof}[Proof of Theorem~\ref{th:Isom}] 

Let $G$ and $G'$ be a pair of graphs such that their filtered algebras $\KK_G$ and $\KK_{G'}$ are isomorphic. Without loss of generality, we can assume that $|E(G)|\leq |E(G')|.$ 
Denote the numbers of vertices  in $G$ and  $G'$ by $n+1$ and $n'+1$   resp.

We consider  $\KK_G$ as a subalgebra in $\Phi_{G}$. The elements $\widetilde{Y}_i=\exp(X_i)-1,\; i\in[0,n]$ form  a set of  generators of  $\KK_G$. Denote by $\widetilde{Z}_i \in \KK_G,\;  i\in[0,n']$  the elements  corresponding to the vertices of $G'$ under the isomorphism of filtered algebras. The set $\{\widetilde{Z}_i, i\in[0,n']\}$ is also a generating set for  $\KK_G$ which  gives the same filtered structure and satisfies the assumptions of  Lemma~\ref{pr-basis}.   
In order to avoid confusion, we call  $\widetilde{Y}_i$  the {\it $i$-th vertex of graph $G$}, and we call  $\widetilde{Z}_j$  the {\it $j$-th vertex of graph $G'$}.

Since $\widetilde{Y}_i,\; i\in[0,n]$ and $\widetilde{Z}_i,\; i\in[0,n']$ determine the same graded structure, then, in particular, 
$$\text{span}\{1,\tY_0,\ldots,\tY_n\}=\text{span}\{1,\tZ_0,\ldots,\tZ_{n'}\}.$$ 
Additionally,  by Lemma~\ref{pr-basis},  $\widetilde{Y}_i,\; i\in[0,n]$ and $\widetilde{Z}_i,\; i\in[0,n']$ are nilpotents, implying that 
$$\text{span}\{\tY_0,\ldots,\tY_n\}=\text{span}\{\tZ_0,\ldots,\tZ_{n'}\}.$$

Firstly, we need to show that each edge-element $\phi_e$ belongs to at most  two different $\widetilde{Z}_i$'s. 
Assume the contrary, i.e., that $\phi_e$ belongs  to $\tZ_i, \tZ_j$ and $\tZ_k$. Then there exist three distinct non-zero coefficients $r_1,r_2,r_3 \in \K$ such that  $\phi_e$ does not belong to $r_1\tZ_i+r_2\tZ_j+r_3\tZ_k$. 
Moreover, for generic distinct non-zero coefficients $r_1',r_2',r_3' \in \K$, element  $\phi_{e'}$ ($e'\in E(G)$) belongs to $r_1'\tZ_i+r_2'\tZ_j+r_3'\tZ_k$
if and only if $\phi_{e'}$ belongs to at least one of $\tZ_i, \tZ_j$ and $\tZ_k$. Hence by Lemma~\ref{lem-degree}, $$d(r_1\tZ_i+r_2\tZ_j+r_3\tZ_k)<d(r_1'\tZ_i+r_2'\tZ_j+r_3'\tZ_k).$$ But at the same time, by Lemma~\ref{pr-basis}~(3), they should coincide, contradiction.

By Lemma~\ref{pr-basis},  for any $i\in[0,n']$, the degree $d(\widetilde{Z}_i)$ equals to the valency of $\widetilde{Z}_i$. Therefore,  
$$2|E(G')|=\sum_i^{n'}d(\widetilde{Z}_i)\leq 2|E(G)|,$$
because each edge-element  is included in at most two $\widetilde{Z}_i$. Since $|E(G)|\leq |E(G')|$,  we conclude that  $|E(G)|= |E(G')|.$ 
Furthermore, by Lemma~\ref{pr-basis} (2), each element $\phi_e,\; e\in E(G)$ belongs exactly to two vertices from $\tZ_i,\; i\in[0,n']$ with the opposite coefficients. Since  $|E(G)|= |E(G')|,$ we can additionally  assume that the number of pairs of non-multiple edges which have a common vertex  in $G'$ is bigger than that in $G$.

So far we have constructed a bijection between the  edges of $G$ and the edges of $G'$. We want to prove that this bijection provides a  graph isomorphism. We will achieve this as a result of the  5 claims collected in the following proposition which is closely related to Lemma~ \ref{gr-edges}. 

\begin{proposition}\label{prop:5claims} The following facts hold.  

\begin{enumerate}
\item {\it If $\phi_{e_1}$ and $\phi_{e_2}$ have no common vertex in $G$, then they  have no common vertex  in~$G'$ as well.}

\item {\it If $\phi_{e_1}$ and $\phi_{e_2}$ are multiple edges in~$G$, then they are multiple edges in~$G'$ as well.}

\item {\it If $\phi_{e_1}$ and $\phi_{e_2}$ have exactly one common vertex in $G$, then they  have exactly one common vertex in~$G'$ as well.}

\item {\it If $\phi_{e_1}$, $\phi_{e_2}$ and $\phi_{e_3}$ form a star in~$G$, then they  form a star in~$G'$ as well.}
(Three edges form a star if they have one common vertex and their three other ends are distinct.)

\item {\it If $\phi_{e_1}$, $\phi_{e_2}$ and $\phi_{e_3}$ form a triangle in~$G$, then they  form a triangle in~$G'$ as well.} 
\end{enumerate}
\end{proposition}

\begin{proof}
To prove (1), assume the contrary, i.e.,  assume that   $\phi_{e_1}$ and $\phi_{e_2}$ belong to $\tZ_j$ (and denote  the corresponding  coefficients by $a$ and $b$ resp.). 
Since elements $\tY_0,\ldots,\tY_n$ have no monomial $\phi_{e_1}\phi_{e_2}$, then $\tZ_0,\ldots,\tZ_{n'}$  have no monomial $\phi_{e_1}\phi_{e_2}$ as well (since their spans  coincide). Then $\ln(1+\tZ_j)$ contains the monomial $\phi_{e_1}\phi_{e_2}$ with the coefficient~$-ab$. 

By Lemma~\ref{pr-basis} (2), we have $\sum_{i=0}^{n'} \ln(1+\tZ_i)$, then there exists $k\in [0,n'], k\neq i$ such that $\ln(1+\tZ_k)$ 
contains the  monomial $\phi_{e_1} \phi_{e_2}$ with a non-zero coefficient.  
Then $\tZ_k$ must contain $\phi_{e_1}$ and $\phi_{e_2}$ (since $\tZ_k$ does not contains  $\phi_{e_1}\phi_{e_2}$). Hence, $\tZ_k$ has $\phi_{e_1}$ and $\phi_{e_2}$ with coefficients $-a$ and $-b$ resp.  
Therefore  $\ln(1+\tZ_k)$ contains monomial $\phi_{e_1}\phi_{e_2}$ with the coefficient $-(-a)(-b)=-ab$. Thus the sum $\sum_{j=0}^{n'} \ln(1+\tZ_j)$ contains $\phi_{e_1}\phi_{e_2}$ with coefficient $-2ab$, contradiction.

\smallskip
To prove (2), consider the map from $\text{span}\{\tY_0,\ldots,\tY_n\}$ to $\K^2$, sending  an element from the span  to the  pair of coefficients of  $\phi_{e_1}$ and $\phi_{e_2}$ resp. 
Since edges $e_1$ and $e_2$ are multiple in $G$, the image of this map has dimension $1$. 
If $\phi_{e_1}$ and $\phi_{e_2}$ are not multiple in~$G'$, then the image of the map  from $\text{span}\{\tZ_0,\ldots,\tZ_{n'}\}=\text{span}\{\tY_0,\ldots,\tY_n\}$ has dimension $2$.

\smallskip
 To prove (3), observe that we have already settled Claims  1 and  2, and also we additionally assumed that the number of pairs of edges which have a common vertex in  $G'$  is bigger than that in $G$. Then each such pair of edges from $G$ is mapped  to the pair of edges from $G'$ with the same property.

\smallskip
To prove (4), consider the map from $\text{span}\{\tY_0,\ldots,\tY_n\}$ to $\K^3$, sending an element in the span   to the  triple of coefficients of  $\phi_{e_1},$ $\phi_{e_2}$ and $\phi_{e_3}$ resp. 
The image of this map has dimension $3$. 
However if $\phi_{e_1}$, $\phi_{e_2}$ and $\phi_{e_3}$ form a triangle in~$G'$, then the image of the map  from $\text{span}\{\tZ_0,\ldots,\tZ_{n'}\}$ has dimension $2$.

\smallskip
Proof of (5) is Â similar to  that of  (4).\end{proof}

Now applying Lemma~\ref{gr-edges} we finish our proof of Theorem~\ref{th:Isom}.  
\end{proof}

\section {Further generalizations}

In this section we will consider the Hilbert series of other filtered algebras similar to $\KK_G$.  (Recall that  the Hilbert series  of a filtered algebra is, by definition, the   Hilbert series  of its associated graded algebra.)

Let $f$ be a univariate polynomial or a formal power series over $\K$. We define the subalgebra $\cF[f]_G\subset \Phi_G$   as  generated by $1$ together with 
$$f(X_i)=f\left(\sum c_{i,e}\phi_e\right),\; i=0,\dots,n.$$ 
\begin{example}
For $f(x)=x$, $\cF[f]_G$ coincides with $\CC_G$. For $f(x)=\exp(x)$, $\cF[f]_G$  coincides with $\KK_G$. 
\end{example}

Obviously, the filtered algebra $\cF[f]_G$ does not depend on the constant term of~$f.$ From now on,  we  assume that $f(x)$ has no constant term, since for any $g$ such that $f-g$ is constant, the filtered algebras $\cF[f]_G$ and $\cF[g]_G$ are the same.

\begin{proposition}
Let $f$ be any polynomial with a non-vanishing linear term. Then algebras $\CC_G$ and $\cF[f]_G$  coincide as subalgebras of $\Phi_G$.
\end{proposition}

\begin{proof}
The argument  is the same as in the proof of Lemma~\ref{lem:isCK}. We  only need to change $\exp(x)-1$ to $f(x)$ and $\ln(1+y)$ to $f^{-1}(y)$.
\end{proof}


\begin{theorem}\label{th:Isom-f} Let $f$ be any polynomial with non-vanishing  linear and quadratic terms. Then given two simple graphs $G_1$ and $G_2,$ 
$\cF[f]_{G_1}$ and $\cF[f]_{G_2}$ are isomorphic as filtered algebras  if and only if  $G_1$ and $G_2$ are isomorphic graphs. 
\end{theorem}

\begin{proof} Repeat the proof of Theorem~\ref{th:Isom}. 
\end{proof}


\subsection{Generic functions $f$ and their Hilbert series}
Since $X_i^{d_i+1}=0$ for any $i$, we can always truncate any polynomial (or a formal power series) $f$ at  degree $|G|+1$ without changing $\cF[f]_G$.
Therefore, for a given graph $G$, it suffices to consider   $f$ as a polynomial of degrees less than or equal to  $|G|$. 
 To simplify our notation,  let us  write  $HS_{f,G}$ instead of $HS_{\cF[f]_G}$. 

\medskip
Given a graph $G$, consider the space of polynomials of degree  less than or equal to $|G|$ and the corresponding Hilbert series. 

\begin{proposition}
\label{pr-hil} In the above notation, for generic polynomials $f$ of degree at most $|G|$, the Hilbert series $HS_{f,G}$ is the same. This generic Hilbert series (denoted by $HS_{G}$ below) is maximal  in the majorization partial order  among all  $HS_{g,G},$ where $g$ runs over the set of all formal power series with non-vanishing linear term. 
\end{proposition}

Recall that, by definition,  a sequence $(a_0,a_1,\ldots)$ is {\it bigger} than $(b_0,b_1,\ldots)$ in the majorization partial order if and only if, for any $k\geq 0$, 
  $$\sum_{i=0}^ka_i\geq \sum_{i=0}^kb_i.$$ More information about the majorization partial order can be found in e.g. \cite{MaOlAr}.
\begin{proof}
Note that,  for a function $f,$ the sum of the first $k+1$ entries of its Hilbert series $HS_{f,G}$
equals the  dimension of $span\{f^{\alpha_0}(X_0)f^{\alpha_1}(X_1)\cdots f^{\alpha_n}(X_n):\ \sum_{i=0}^n \alpha_i\leq k\}.$ It is obvious that, for a generic $f,$ this dimension is maximal. Since all Hilbert series $HS_{f,G}$ are polynomials of degree at most  $|G|+1,$ then the required property has to be checked only for $k\leq [G|$. Therefore it is obvious that, for generic $f$, their Hilbert series is  maximal in the majorization order.  
\end{proof}

\begin{remark}ÃÂ 
{\rm We know that the Hilbert series of the graded algebra $\CC_G$ is a specialization of the Tutte polynomial of $G$. 
However we can not  calculate the Hilbert series of $\KK_G$ from the Tutte polynomial of $G$, because there exists a pair of graphs 
$(G,G')$ with the same Tutte polynomial and different $HS_{\KK_{G}}$ and $HS_{\KK_{G'}}$, see Example~\ref{examp}.

Additionally, notice that, in general,  $HS_{\exp,G}:=HS_{\KK_G} \neq HS_{G}$. 
Analogously we can not calculate generic 
 Hilbert series $HS_G$ from the Tutte polynomial of $G$, 
 see Example~\ref{examp}.}
 \end{remark}

\begin{example} \label{examp}
{\begin{figure}[htb!]

\centering
\includegraphics[scale=0.5]{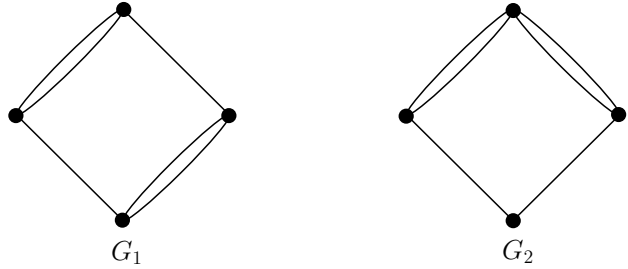}

\captionsetup{justification=centering}
\caption{Graphs with the same matroid and different ``K-theoretic" and generic Hilbert series.}
\label{g1g2}
\end{figure}

}
{\rm  Consider two graphs $G_1$ and $G_2$ presented in Fig.~\ref{g1g2}. It is well-known that $G_1$ and $G_2$ have isomorphic matroids and  hence, the same Tutte polynomial. Therefore, the Hilbert series of $\CC_{G_1}$ and $\CC_{G_2}$ coincide. Namely, 
$$HS_{\CC_{G_1}}(t)=HS_{\CC_{G_2}}(t)=1+3t+6t^2+9t^3+8t^4+4t^5+t^6.$$ 
However, the Hilbert series of ``K-theoretic" algebras are distinct. Namely 
$$HS_{\KK_{G_1}}(t)=1+4t+10t^2+14t^3+3t^4,$$ 
$$HS_{\KK_{G_2}}(t)=1+4t+10t^2+15t^3+2t^4.$$ 
Moreover their generic Hilbert series are also distinct and different from  their ``K-theoretic" Hilbert series. Namely, 
$$HS_{{G_1}}(t)=1+4t+10t^2+15t^3+2t^4,$$ 
$$HS_{{G_2}}(t)=1+4t+10t^2+16t^3+t^4.$$
Putting our information together we get, $$HS_{\CC_{G_1}}=HS_{\CC_{G_2}}\prec HS_{\KK_{G_1}}\prec  HS_{\KK_{G_2}}=HS_{{G_1}}\prec  HS_{{G_2}},$$ where $\prec$ denotes the majorization partial order.} 
\end{example}\

\section{``K-theoretical" analog for  spanning trees}

 For an arbitrary loopless graph $G$  on the vertex set $\{0, . . . , n\},$  
let  $\Phi_G^T$ be the graded commutative algebra over a given field $\K$ generated by the variables $\phi_e, e \in G$, with the defining relations:
$$(\phi_e)^2 = 0, \quad \text {for any edge}\; e\in G;$$
$$\prod_{e\in H}\phi_e=0, \quad \text {for any non-slim subgraph\;\;} H\subset G,$$
where a subgraph $H$ is called {\it slim} if its complement $G\setminus H$ is connected.

Let $\CC_G^T$ be the subalgebra of  $\Phi_G^T$ generated by the elements
$$X_i^T =\sum_{e\in G} c_{i,e} \phi_e,$$
for  $i = 1, . . . , n, $ where $c_{i,e}$ is given by \eqref{eq:def}. (Notice that $X_i^T$ and $X_i$ are defined by exactly the same formula but in different ambient algebras.) 

Algebra  $\CC_G^T$ will be called the {\it spanning trees counting algebra}ÃÂ of $G$ and is,  obviously, the  quotient of $\CC_G$ modulo the set of relations $\prod_{e\in H}\phi_e=0$ over all non-slim subgraphs $H$. Its defining set of relations is very natural and resembles that of \eqref{eq:relF}. Namely, define  the ideal  $\J_G^T$ in $\K[x_1,\dots,x_n]$ as generated by the polynomials: 
\begin{equation}\label{eq:relT}
p_I^T=\left( \sum_{i\in I}x_i\right)^{D_I},
\end{equation} 
where $I$ ranges over all nonempty subsets in $\{1,\dots,n\}$ and the number $D_I$ is the same as in 
\eqref{eq:relF}. Set $\B^T_G:=\K[x_1,\dots,x_n]/\J_G^T.$ One of the results of \cite {PS} claims the following.

\begin{theorem}\label{th:trees} {\rm [Theorems~9.1 and Corollary~10.5 of \cite{PS}]} For any simple graph $G$ on the set of vertices $\{0,1,\dots,n\},$  algebras  $\B_G^T$ and $\CC_G^T$ are isomorphic. 
Their total dimension is equal to the number of spanning trees in  $G$.
The dimension $\dim B_G^T(k)$ of the $k$-th graded component of  $\B_G^T$ equals
the number of spanning trees $T$ in $G$ with external activity $|G|-n-k$. \end{theorem} 

Similarly to the above,  we can  define the filtered algebra $\KK_G^T\subset \Phi_G^T$   which is isomorphic to $\CC_G^T$ as a non-filtered algebra.   Namely, $\KK_G^T$ is defined by the generators:
$$Y_i^T=\exp(X_i^T)=\prod_{e\in G} (1+c_{i,e}\phi_e),\; i=0,\dots,n.$$

\medskip
The first result of this section is as follows.  Define  the ideal  $\I_G^T \subseteq \K[y_0,y_1,\dots,y_n]$ as generated by the polynomials: 
\begin{equation}\label{eq:relTT}
q_I^T=\left( \prod_{i\in I}y_i-1\right)^{D_I},
\end{equation} 
where $I$ ranges over all nonempty proper subsets in $\{0,1,\dots,n\}$ and the number $D_I$ is the same as in 
\eqref{eq:relF}, together with the generator
\begin{equation}\label{eq:relTT-full}
q_{\{0,1,\dots,n\}}^T= \prod_{i=0}^n y_i-1.
\end{equation} 
 Set $\cD_G^T:=\K[y_0,\dots,y_n]/\I_G^T.$

\medskip
 We present two  results similar to the case  of spanning  forests. 

\begin{theorem}\label{th:4alg-tree} For any simple graph $G,$  algebras   $B_G^T,$  $\CC_G^T$, $\cD_G^T$ and $\KK_G^T$  are  isomorphic as (non-filtered) algebras.
Their total dimension is equal to the number of spanning trees in $G$. 
\end{theorem}
\begin{proof}
The proof is similar to  that of Theorem~\ref{th:4alg}. 
Algebras $\CC_G^T$ and $\KK_G^T$  coincide as subalgebras of $\Phi_G^T$ (but they have different filtrations); 
algebras $\CC_G^T$ and $B_G^T$ are isomorphic by Theorem~\ref{th:trees}. 
The proof of the isomorphism between $\cD_G^T$ and $B_G^T$ is the same as above; we  only need to add the variable $x_0=-(\sum_{i=1}^nx_i)$ to $B_G^T$. 
\end{proof}

\begin{theorem}\label{th:main-tree} For any simple graph $G$,  algebras  $\cD_G^T$ and $\KK_G^T$ are  isomorphic  as filtered algebras. 
\end{theorem}
\begin{proof}
Similar to the above proof of Theorem~\ref{th:main}.
\end{proof}

 \medskip
To move further,  we need to give a definition.

\begin{definition} Let $G$ be a connected graph. We define its $\Delta$-subgraph $\widehat{G}\subseteq G$ as the subgraph  with following edges and vertices:
\begin{itemize}
\item $e\in E(\widehat{G}),$ if $e$ is not a bridge (i.e., $G\setminus e$ is still connected),
\item $v\in V(\widehat{G}),$ if there is an edge $e\in E(\widehat{G})$ incident to  $v$.
\end{itemize}
\end{definition}

In general, $\widehat{G}$ contains more information about $G$  than its matroid, because there exist graphs with isomorphic matroids and non-isomorphic $\Delta$-subgraphs, see Figure~\ref{delta-g1g2}.

Recall that in a recent paper \cite{Ne},  the first author has shown that $\CC_G^T$ depends only on the bridge-free matroid of $G$. Namely, 

\begin{proposition}\label{nenashev-tree}{\rm [Proposition~12 of \cite{Ne}]}  For any two connected graphs $G_1$ and $G_2$ with isomorphic bridge-free matroids (matroids of their $\Delta$-subgraphs),  algebras $\CC_{G_1}^T$ and $\CC_{G_2}^T$ are isomorphic.
\end{proposition}

 Unfortunately, we can not prove the converse implication at present although we conjecture that is should hold as well, see Conjecture~\ref{cut-spaces} in \S~\ref{sec:final}. 
 In case of filtered algebra $\KK_{G_1}^T$ and $\KK_{G_2}^T$ we  can also prove an appropriate result only in one direction, see Proposition~\ref{pr:Isom-tree}. 

\medskip
Similarly to \S~3 we can to define  $\cF[f]^T_G\subset \Phi_G$.
Let $f$ be a univariate polynomial or a formal power series over $\K$. We define the subalgebra $\cF[f]^T_G\subset \Phi_G$   as  generated by $1$ and by
$$f(X_i^T)=f\left(\sum c_{i,e}\phi_e\right),\; i=0,\dots,n.$$

\begin{proposition}\label{pr:Isom-tree} For univariate polynomial $f$ and any two 
 connected graphs $G_1$ and $G_2$ with isomorphic $\Delta$-subgraphs $\widehat{G}_1$ and $\widehat{G}_2$, algebras $\cF[f]^T_{G_1}$ and $\cF[f]^T_{G_2}$   are isomorphic as filtered algebras.
 Additionally, $\KK_{G_1}^T$ and $\KK_{G_2}^T$ are isomorphic as filtered algebras.
\end{proposition}
\begin{proof}
Note that if $G$ has a bridge $e$, then filtered algebra $\cF[f]^T_{G}$  is the Cartesian product of filtered algebras $\cF[f]^T_{G'}$  and $\cF[f]^T_{G''}$ , where $G'$ and $G''$ are connected component of $G\setminus e$.

Thus filtered algebra $\cF[f]^T_{G}$  is the Cartesian product of such filtered algebras corresponding to the connected components of the $\Delta$-subgraph of $G$.

Therefore  if connected graphs $G_1$ and $G_2$ have isomorphic $\Delta$-subgraphs, then their filtered algebras $\cF[f]^T_{G_1}$  and $\cF[f]^T_{G_2}$  are isomorphic.
\end{proof}

\begin{remark}\label{rem-tree}{\rm 
In general case we can not prove that these algebras distinguish graphs with different $\Delta$-subgraphs. The proof of Theorem~\ref{th:Isom} does not work for two reasons.
Firstly,   $d(\tY_i)$ is not the degree of the $i$-th vertex in $G$. Secondly,  even if we can construct a similar bijection between edges, we do  not have  an analog of Proposition~\ref{prop:5claims}. Since in the proof we consider coefficients  of monomial 
$\phi_{e_1}\phi_{e_2}$, in case when ${e_1}$ and ${e_2}$ are not bridges and when $\{e_1,e_2\}$ is a cut, this monomial can still lie  in the ideal.

It is possible to construct such a bijection in a smaller set of graphs, namely for graphs such that, for any edge $e$ in the graph, there is another edge $e'$   which is multiple  to $e$. For such graphs we do  not have the second problem, because if $\{e_1,e_2\}$ is a cut, then $e_1$ and $e_2$ are multiple edges.
So, instead of the actual converse of Proposition~\ref{pr:Isom-tree}, we can prove the converse in the latter  situation, but we do not present this result here. 
}
\end{remark}

\smallskip
\begin{proposition}
\label{pr-hil-tree} In the above notation, for generic polynomials $f$ of degree at most $|G|$, the Hilbert series $HS_{\cF[f]^T_G}$ is the same. This generic Hilbert series (denoted by $HS_{{G}^T}$ below) is maximal  in the majorization partial order among  $HS_{\cF[g]^T_G}$ for $g$ running over the set of power series
 with non-vanishing linear term.
\end{proposition}
\begin{proof}
See the proof of Proposition~\ref{pr-hil}.
\end{proof}

\begin{example}
\label{delta}

{\rm Consider two graphs $G_1$ and $G_2$, see Fig.~\ref{delta-g1g2}. It is easy to check that subgraphs $\widehat{G}_1$ and $\widehat{G}_2$ have isomorphic matroids, implying that algebras $\CC_{G_1}^T$ and $\CC_{G_2}^T$ are isomorphic.

{\begin{figure}[htb!]

\centering
\includegraphics[scale=0.5]{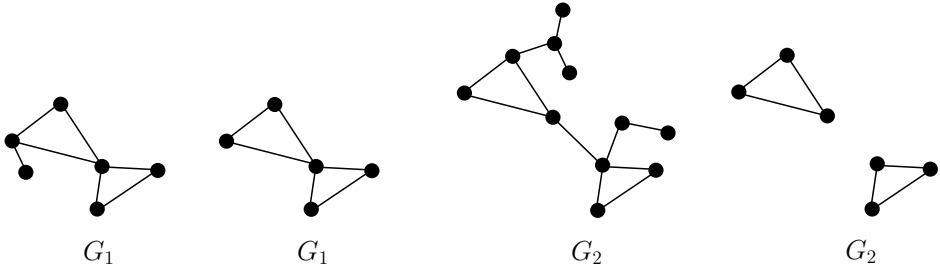}
\captionsetup{justification=centering}
\caption{Graphs and their $\Delta$-subgraphs.}
\label{delta-g1g2}
\end{figure}
}

$$HS_{\CC_{G_1}^T}(t)=HS_{\CC_{G_2}^T}(t)=1+4t+4t^2.$$ 
The Hilbert series of ``K-theoretic" algebras are distinct, namely 
$$HS_{\KK_{G_1}^T}(t)=1+5t+3t^2,$$ 
$$HS_{\KK_{G_2}^T}(t)=1+6t+2t^2.$$ 
These graphs are ``small", so their generic Hilbert series  coincides with the ``K-theoretic" one. 
Putting our information together, we get $$HS_{\CC_{G_1}^T}=HS_{\CC_{G_2}^T}\prec HS_{\KK_{G_1}^T}=HS_{{G_1}^T}\prec  HS_{\KK_{G_2}^T}=HS_{{G_2}^T}.$$} 
\end{example}\

\section {Related problems.}\label{sec:final}

 At first, we formulate several  problems in case of spanning forests;  their analogs for spanning trees are straight-forward.
 
 \begin{problem}
For which functions $f$ besides $a+b x$ and  $a+b e^x$, one can present relations in $\cF[f]_{G}$  for any graph $G$ in a simple way? In other words,  for which $f,$  one can define an algebra  similar to $B_G$ and $\cD_G$? 
\end{problem}

Since the Hilbert series  $HS_{\KK_{G}}$ and $HS_{G}$ are not expressible in terms of the Tutte polynomial of $G$,  they contain some other   information about $G$.

\begin{problem}
 Find combinatorial description of $HS_{\KK_{G}}$ and $HS_{G}$? 
\end{problem}

\begin{problem}
For which graphs $G$, Hilbert series $HS_{\KK_{G}}$ and $HS_{G}$  coincide? In other words, for which $G$, $\exp$ is a generic function?
\end{problem}

\begin{problem}
Describe combinatorial properties of $HS_{f,G}$  when $f$ is a function starting with a monomial of degree bigger than $1$, i.e. $f(x)=x^k+\cdots,\ k>1$? 
In particular, calculate the  total dimension of $\cF[f]_G$. 
\end{problem}

The most delicate and intriguing question is as follows. 

\begin{problem}
Do there exist non-isomorphic graphs $G_1$ and $G_2$ such that, for any polynomial $f(x)$, the Hilbert series  $HS_{f,G_1}$ and $HS_{f,G_2}$ coincide? In other words, does the collection of Hilbert series $HS_{f,G}$ taken over all formal series $f$ determine $G$ up to isomorphism?
\end{problem}

The following problems deal with the case of spanning trees only.

\begin{conjecture}{\rm [comp.~\cite{Ne}]}
\label{cut-spaces}
 Algebras $\CC_{G_1}^T$ and $\CC_{G_2}^T$ for graphs $G_1$ and $G_2$ are isomorphic if and only if 
their bridge-free matroids are isomorphic, where the bridge-free matroid is the graphical matroid of $\Delta$-subgraph. 

\end{conjecture}

\begin{problem}
Which class  of graphs satisfies the property that if two graphs $G_1$ and $G_2$ from this class have  isomorphic 
$\KK_{G_1}^T$ and $\KK_{G_2}^T$, then their $\Delta$-subgraphs are isomorphic.
In other words, can one classify all pairs $(G_1,G_2)$ of connected graphs, which has isomorphic filtered algebras $\KK_{G_1}^T$ and
$\KK_{G_2}^T$?  {\rm(}The same problem for $\cF[f]_{G_1}^T$ and
$\cF[f]_{G_2}^T$, where $f(x)=x+ax^2+\cdots${\rm)}
\end{problem}

\end{document}